\documentclass[reqno]{amsart}

\usepackage{amssymb,amsmath,amscd,amsthm}
\usepackage{cite}
\usepackage{graphicx}
\graphicspath{ {./images/} }
\usepackage{float} 
\usepackage{graphicx,psfrag}

\newtheorem{theorem}{Theorem}[section]

\newtheorem{lemma}[theorem]{Lemma}
\newtheorem{proposition}[theorem]{Proposition}

\font\bbc=msbm10 scaled 1200

\newcommand{\R}{\mbox {\bbc R}}

\newcommand{\Lim}[1]{\raisebox{0.5ex}{\scalebox{0.8}{$\displaystyle \lim_{#1}\;$}}}

\date{}
\begin{document}
\title{Non-uniqueness for the ab-family of equations}
  \author{
    John Holmes \and
Rajan Puri }\thanks{Department
 of Mathematics and Statistics,  Wake Forest University,
 Winston Salem, NC 27109, USA, (holmesj@wfu.edu)
 }
 
 \thanks{Department
 of Mathematics and Statistics, Wake Forest University,
 Winston Salem, NC 27109, USA, (purir@wfu.edu). 
}

\date{}


\begin{abstract} 
We study the cubic ab-family of equations, which includes both the Fokas-Olver-Rosenau-Qiao (FORQ) and the Novikov (NE) equations. For $a\neq0$, it is proved that there exist initial data in the Sobolev space $H^s$, $s<3/2$, with non-unique solutions. Multiple solutions are constructed by studying the collision of 2-peakon solutions. Furthermore, we prove the novel phenomenon that for some members of the family, collision between 2-peakons can occur even if the ``faster" peakon is in front of the ``slower'' peakon. 
\end{abstract}



\maketitle

\section{Introduction} We consider the Cauchy problem for the  ab-equation 
    \begin{align}\label{eq1}
        &u_t+u^2u_x-au_x^3+D^{-2}\partial_x\big[\frac{b}{3} u^3+\frac{6-6a-b}{2}uu_x^2\big]+D^{-2}\big[\frac{2a+b-2}{2}u_x^3\big]=0
        \\
        &u(x,0) = u_0(x), \label{eq1-data}
    \end{align}
    where $u_0(x) \in H^s (\mathbb R)$, for $s<3/2$. 
   The two parameters   $a$, $b\in \mathbb R$ and we will assume $a\neq 0$. $u_t$ 
    and $u_x$ denote the derivatives of $u$ with respect to $t$ and $x$, $\partial_x$ 
    denotes differentiation with respect to $x$, and the non-local operator 
    $D^{-2} = (1-\partial_x^2)^{-1}$ 
is the inverse Fourier transform of $(1+\xi^2)^{-1} $.

    Peakons, or peaked traveling wave solutions, were discovered in 1978 by Fornberg and Whitham \cite{FW} and then by Camassa and Holm [CH] in their quest for a water wave model that could capture wave breaking. These are special traveling wave solutions which  take the form 
    $$
    u(x,t) = ce^{-|x-q(t)|}.
    $$
     The Camassa-Holm (CH) equation 
    \begin{align}
    (1-\partial_x)^2 u_t = uu_{xxx} +2u_x u_{xx} - 3uu_x,
    \end{align}
    and later the Degasperis-Procesi (DP) equation 
    \begin{align}
    (1-\partial_x)^2 u_t = uu_{xxx} +3u_x u_{xx} - 4uu_x,
    \end{align}
    introduced in \cite{DP} are two such integral equations which admit peakon solutions. In fact, these equations admit special 2-peakon solutions, called peakon-antipeakon solutions which are of the form 
    $$
    u(x,t) = p(t) e^{-|x+q(t) |} - p(t) e^{-|x-q(t)|}, 
    $$
    where $p(t)$ and $q(t)$, magnitude and position, satisfy a system of ordinary differential equations (ODEs). For a discussion of peakon and other solutions, we refer the reader to Holm
and Ivanov \cite{HI}. Himonas, Holliman and  Grayshan \cite{HHG} show that the peak and antipeak move towards each other, collide in finite time and ill-posedness of the corresponding Cauchy problems in $H^s$, $s<3/2$ is established due to this event. This symmetry does not exist in the ab-family of equations. 
    
    Two members of ab-family (\ref{eq1}) have been studied extensively in different contexts by several researchers. In particular, the choice of parameters $a = 1/3$, $b = 2$ corresponds to the Fokas-Olver-Rosenau-Qiao (FORQ) equation
    \begin{equation}\label{eq2}
        u_t+u^2u_x-u_x^3+D^{-2}\partial_x\big[\frac{2}{3} u^3+uu_x^2\big]+D^{-2}\big[\frac{1}{3}u_x^3\big]=0,
    \end{equation}
    derived in Fokas \cite{F}, Fuchssteiner \cite{Fu}, Olver and Rosenau \cite{OR}, and Qiao \cite{Q};
    while the choice a = 0, b = 3 gives the Novikov equation (NE)
    \begin{equation}\label{eq3}
        u_t+u^2u_x+D^{-2}\partial_x\big[\frac{1}{3} u^3+\frac{3}{2}uu_x^2\big]+D^{-2}\big[\frac{1}{2}u_x^3\big]=0,
    \end{equation}
    derived by Novikov \cite{N}.
    Both of these equations are integrable equations with a bi-Hamiltonian structure, Lax pair, an infinite hierarchy of hereditary symmetries. and infinitely many conserved quantities, for a more detailed discussion of integrability see \cite{Q1} and the references contained therein. In particular, classical solutions to both equations conserve the $H^1$ norm.
        
   The well-posedness theory for the ab-family is not completely understood. Some partial results are as follows. Himonas and Mantzavinos \cite{HMA} showed that FORQ is well-posed in $H^s$, with $s>5/2$ and this was extended to a four-parameter family in \cite{HMA2} which includes the $ab$-family. A nonuniqueness result by Himonas and Holliman \cite{AH} showed that the FORQ equation is ill-posed in $H^s$ for any $s<3/2$. There is no theory concerning well-posedness in the gap $3/2\le s\le 5/2$. In contrast, the NE is well-posed in $H^s$ for all $s>3/2$, see \cite{BO} for details, and Himonas, Kenig and Holliman \cite{HKM} showed ill-posedness in $H^s$ for $s<3/2$. Both the ill-posedness results for the NE equation and the FORQ equation study the behavior of the solution near the time of collision of a 2-peakon solution.  
The method used to prove ill-posedness for both the NE and FORQ equations is similar to that used by many authors for other nonlinear evolution equations. For instance, Bourgain and Pavlovic \cite{BP} studied the Navier-Stokes equations in Besov spaces and Christ, Colloander and Tao \cite{CCT} used a similar technique to study defocusing dispersive PDEs. 
 We too use this idea to prove our main theorem, which is stated as follows. 
    \begin{theorem}\label{main-thm}
For all $b \in \mathbb{R}$ and $a \neq 0$, solutions to the Cauchy problem for the ab-equation are not unique in $H^s$ when $ s<3/2$.
\end{theorem}

Perhaps the most interesting phenomenon discovered in our proof, is that multipeakon solutions to the ab-equation interact unlike classical solitons. The Korteweg-de Vries (KdV) equation is the canonical example of a nonlinear PDE supporting multi-soliton solutions. When two solitons collide, although they interact in a nonlinear way, they decouple; both before and after the collision, they act similarly to a linear superposition of two solitons. Moreover, since their speed is proportional to their height, if a larger soliton begins in front of a smaller soliton, they will never collide. For more information describing their interaction, see Benes, Kasman and Young \cite{BKY}. 

A similar phenomenon has been observed of the 2-peakon solutions of Camassa-Holm type equations.
Shown below is a numerical solutions to the  FORQ equation, with the amplitudes of the two peaks shown in dotted and dashed, while the distance between the peaks is shown in solid. Numerically, the two peakons pass through each other relatively undisturbed, as shown in Figure \ref{fig:forqsln}. (However, we note that for all of the 2-peakon solutions we consider, since solutions are not unique, there are other solutions which are not described by the numerical solution.)  In contrast, for some values of $a$ and $b$, our 2-peakon solutions are entangled and do not seem to separate; we prove that smaller solitons can begin behind larger solitons, and yet, they collide.  In Section 2, we contrast the interaction of the 2-peakon solutions of the ab-equation with the solution of the FORQ equation. 
\begin{figure} [h!]
   \begin{minipage}{1
  \textwidth}
     \centering
\includegraphics[width=0.5\textwidth]{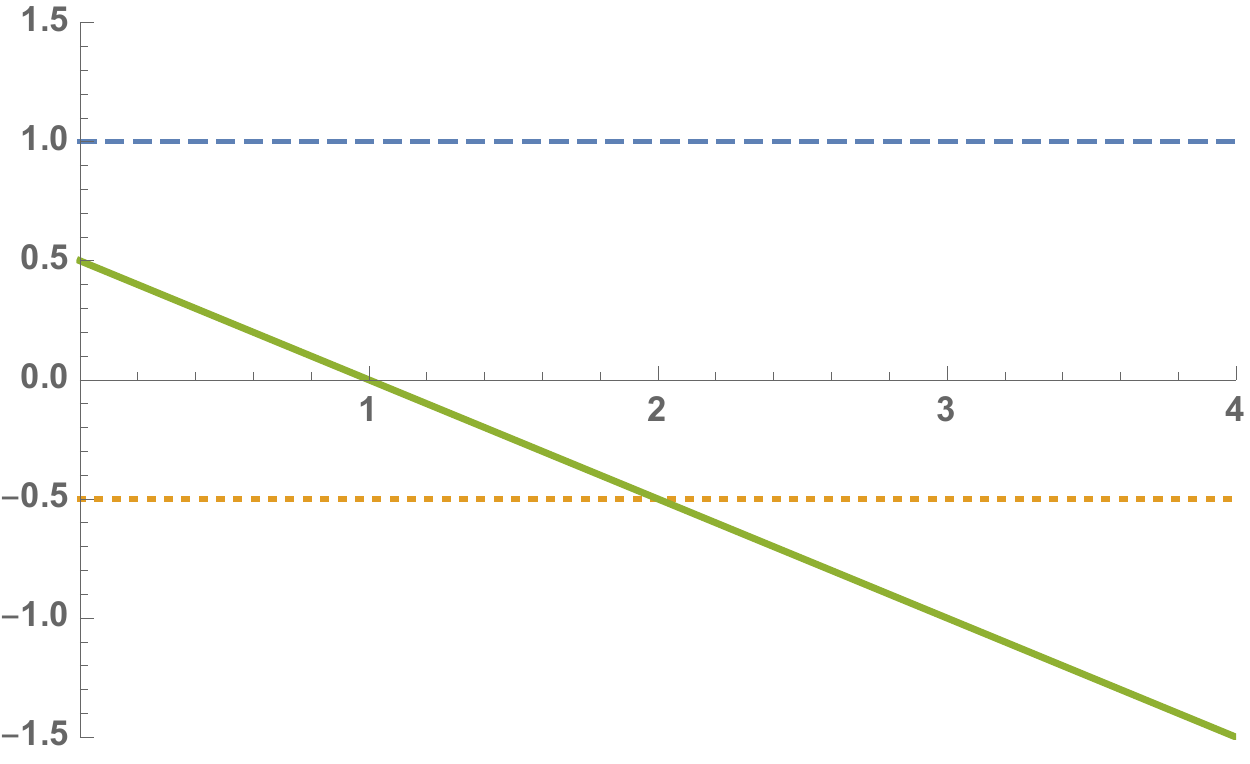}  
\caption{FORQ wave heights (dotted and dashed) and relative positions (solid).}
\label{fig:forqsln} 
   \end{minipage}\hfill
 
\end{figure}

In the next section, we will begin by studying a system of ODEs which describes the time evolution of a 2-peakon solution to the ab-equation. We will show that the solution to the ODE system does not blow up as the two peakons collide. Then, we will show that at the time of collision, the 2-peakon solution reduces to a single peakon solution, which will complete the proof of nonuniqueness. 
 
 \section{The ODE System}
 We begin by recalling the relationship between the 2-peakon solutions of the ab-equation and a system of ODEs. We will study the system of ODEs, prove some preliminary results necessary in the proof of ill-posedness,  then choose the initial data which we can show converges to a single peakon solution. 
\begin{lemma}{ \cite{Ama} } The   2-peakon  function
\begin{equation}\label{eq4}
        u(x,t)=p_1(t)e^{-|x-q_1(t)|}+ p_2(t)e^{-|x-q_2(t)|},
    \end{equation}
    solves the ab-equation if and only if the positions $q_1, q_2$ and the momenta $p_1, p_2$ satisfy the $4\times 4$ system of ordinary differential equations:
     \begin{equation} \label{eq5}
    \begin{split}
    q_1^{'} & = (1-a)p_1^2+2p_1p_2e^{-|q_1-q_2|}+(1-3a)p_2^2e^{-2|q_1-q_2|}\\
    q_2^{'}& = (1-a)p_2^2+2p_1p_2e^{-|q_1-q_2|}+(1-3a)p_1^2e^{-2|q_1-q_2|}\\
    p_1^{'}& = (2-b)sgn(q_2-q_1)p_1p_2e^{-|q_1-q_2|}(p_1+p_2e^{-|q_1-q_2|})\\
    p_2^{'}& = (2-b)sgn(q_1-q_2)p_1p_2e^{-|q_1-q_2|}(p_1e^{-|q_1-q_2|}+p_2).
    \end{split}
    \end{equation}
\end{lemma}
We first observe that if at time zero, $p_1(0), p_2(0) \neq0$, and $q_1 (0) \neq q_2(0) $ (ie we have a multi-peakon initial profile), then, due to the time reversibility of solutions,  nonuniqueness will be shown if either $p_1(t) $ or $p_2(t)$ equal zero at some time, or if $q_1(t) = q_2(t) $ while all the functions remain bounded.

 We define 
$
q(t) \  \dot= \ q_2 - q_1,
$
 $h \  \dot= \ p_2 - p_1$, $w\dot = p_1 +p_2$ 
and $z \  \dot= \ p_1 p_2$. We will assume that at initially, $q_1 (0)= 0$ 
and $q_2(0) = \mu$, and therefore, by continuity,  $q(t) \ge 0$ for some 
$0 \le t \le T^c$, where $T^c $ is the {\em collision} time, i.e when $q (t) = 0$ for the first time. 
    From equation (\ref{eq5}), we have 
    \begin{equation}\label{eq6}
    \begin{split}
     q^{'}=q_2^{'}- q_1^{'} & = \bigg[(1-a)p_2^2+2p_1p_2e^{-|q_1-q_2|}+(1-3a)p_1^2e^{-2|q_1-q_2|}\bigg]\\
                          & -\bigg[(1-a)p_1^2+2p_1p_2e^{-|q_1-q_2|}+(1-3a)p_2^2e^{-2|q_1-q_2|}\bigg]\\
                      &   =(1-a)(p_2^2-p_1^2)-(1-3a)(p_2^2-p_1^2)e^{-2|q_1-q_2|}\\
                      &= (p_2-p_1)(p_2+p_1)\bigg((1-a)-(1-3a)e^{-2|q_1-q_2|}\bigg).
    \end{split}
    \end{equation}
  Additionally, the difference  $p_2^{'}- p_1^{'}$ and $p_2^{'}+ p_1^{'}$ satisfy
       \begin{equation}\label{eq8}
    \begin{split}
    h^{'}=p_2^{'}- p_1^{'} & =\bigg[(2-b)sgn(q_1-q_2)p_1p_2e^{-|q_1-q_2|}(p_1e^{-|q_1-q_2|}+p_2)\bigg]\\ 
                    & -\bigg[(2-b)sgn(q_2-q_1)p_1p_2e^{-|q_1-q_2|}(p_1+p_2e^{-|q_1-q_2|})\bigg]\\
                      &   =-(2-b) p_1p_2e^{-|q_1-q_2|}(p_1+p_2+(p_1+p_2)e^{-|q_1-q_2|})\\
                      & =-(2-b)p_1p_2(p_1+p_2)(1+e^{-|q_1-q_2|})e^{-|q_1-q_2|},
    \end{split}
    \end{equation}
    and,   \begin{equation}\label{eq9} 
    w^{'}= (2-b)p_1p_2(p_1-p_2)(1-e^{-|q_1-q_2|})e^{-|q_1-q_2|} . 
    \end{equation}
Therefore, from the above calculation, we deduce 
\begin{align}\label{qhwz-system}
\begin{aligned}[c]
q' &=  hw( 1-a  -(1-3a)   e^{-2q} ),\\
h' &=- (2-b) wz(1+e^{-q}) e^{-q} ,\\
w' &=-(2-b) hz (1-e^{-q} ) e^{-q} ,\\
z' &= (2-b) hwz e^{-2q} ,
\end{aligned}
\qquad
\begin{aligned}[c]
q(0)&= q_2(0) - q_1(0) =  \mu   ,\\
h(0)&= p_2(0) - p_1(0) = h_0,\\
w(0) &= p_1(0) + p_2(0) = w_0, \\
z(0)&= p_1(0)  p_2(0) = z_0.
\end{aligned}
\end{align}
If  $a=1/3$, then $( 1-a  -(1-3a)   e^{-2q} ) =2/3$, and we set $\mu>0$ to be a  small positive number. Else, we choose a number $1<c<2$ (there is always a choice of $c$) such that 
$$
0<  \mu = \mu(a) = -\frac{1}{2} \ln \left( \frac{ (c+1)  a - 1}{3a-1} \right)  <1 ,
 $$
which implies 
 $$
L_a(\mu) \ \dot = \ 1 - e^{-2\mu} + 3a e^{-2 \mu} -a  = c a .
$$
We note that as $ q$ tends to $0$, $( 1-a  -(1-3a)   e^{-2q} ) $ tends to $2a$, and therefore, the sign of this term remains constant. Also, note that $L_a(\mu) $ can never take the value $0$ when $a\neq 0$ on the domain $ \mu \ge 0$. 
The main results will require different initial data  for $h$, $w$ and $z$   depending upon both $a$ and $b$. Therefore, we hold off on the choices of $p_1(0)$ and $p_2(0)$ until necessary.  
Our first proposition describes the solution to the above system. In particular, we will see that so long as $q \geq 0 $, the functions $h, w, $ and $z$ are bounded and therefore $p_1 = \frac12(w-h)$ and $p_2 = \frac12(h+w)$ remain bounded as well. 

\begin{proposition} 
If  $a \neq 0 $, for any real numbers $h_0, w_0, z_0$, the initial value problem \eqref{qhwz-system} has a unique smooth solution on some positive time interval. Furthermore, in terms of $q$, the functions $h(t )$ and $w(t)$ and $z(t)$ remain  bounded for all $0\le q $.

\end{proposition}
\begin{proof}
The right hand side of the  system \eqref{qhwz-system} is smooth in the arguments $q,h,w,z$ and therefore, by the ODE theorem, has a solution on some time interval $[0, T)$, $T>0$.  
We now derive the relations between $h, w, z$ and $q$. Beginning with $z$, from the equations with $z'$ and $q'$ we find 
\begin{align*}
\frac{z'}{z} = \frac{ (2-b) e^{-2q}  q'}
{  1-a-(1-3a)  e^{-2q} }.
\end{align*}
Therefore, we have 
\begin{align*}
\frac{d}{dt} \ln|z|  =  \frac{(2-b)}{2(1-3a)}  \frac{d}{dt } \ln |1-a-(1-3a)  e^{-2q} | .
\end{align*}
If $z_0 >0$, then by continuity, we can assume $z>0$ for some time and $|z| = z$. Likewise, if $z_0 < 0$, we will assume $|z| = -z$. In either case, integrating from $0$ to $t$ yields 
\begin{align}
\ln \left( \frac{z(t)}{z_0} \right) =  \frac{(2-b)}{2(1-3a)} \ln \left(  \frac{ 1-a-(1-3a)  e^{-2q} }{  1-a-(1-3a)  e^{-2\mu} }  \right) .
\end{align} 
We exponentiate and rearrange terms to find 
\begin{align}\label {z-eq}
z(t) =    z_0  \left(  \frac{ 1-a-(1-3a)  e^{-2q} }{ 1-a-(1-3a)  e^{-2\mu} }  \right)^{  \frac{(2-b)}{2(1-3a)}} .
\end{align}
 From here we can conclude that $z(t) $ remains bounded for all $ t \in [0, T^c]$, since neither the numerator nor the denominator take the value $0$.

 Now, we  will use the above formula for $z(t)$ to find $h(t)$. We have 
 \begin{align*}
 \frac{h'}{q'} =  \frac{- (2-b) wz(1+e^{-q}) e^{-q} } { hw( 1-a  -(1-3a)   e^{-2q} )}, 
 \end{align*} 
 or rearranging we find 
 \begin{align*}
 h h'  =  z\cdot  \frac{- (2-b)  (1+e^{-q}) e^{-q} q' } {   1-a  -(1-3a)   e^{-2q} } .
 \end{align*} 
 We substitute the formula found for $z$ in equation \eqref {z-eq} to get
  \begin{align*}
 h h'  = z_0  \left|  1-a-(1-3a)  e^{-2\mu}   \right| ^{  \frac{-(2-b)}{2(1-3a)}}   \frac{- (2-b)  (1+e^{-q}) e^{-q}  q'} {  | 1-a  -(1-3a)   e^{-2q} |
 ^{1- \frac{(2-b)}{2(1-3a)}}} .
 \end{align*} 
Set 
  \begin{align}\label{f-def} 
f(q) \ \dot = \ z_0  \left|  1-a-(1-3a)  e^{-2\mu}   \right| ^{  \frac{-(2-b)}{2(1-3a)}}   \frac{- (2-b)  e^{-q}   } {  | 1-a  -(1-3a)   e^{-2q} |
 ^{1- \frac{(2-b)}{2(1-3a)}}}, 
 \end{align} 
and then define 
$$
 F_1 (q) \ \dot = \ \int_0^t   (1+e^{-q}) f(q)  dq  . 
 $$
 Since $f(q)$ is smooth and bounded for all $q \ge 0$ (since the denominator is singular only when $a = 0$), $F(q) $ remains smooth, bounded  and differentiable for all $0\le t\le T^c$. 
 Therefore, 
 $$
 h^2   = h_0^2 + 2 F_1(q) ,
 $$
 remains bounded for all $0\le t \le T^c$.

 Next we solve for $w(t)$. We rearrange the equations for $w'$ and $q'$ to find 
\begin{align*}
ww' =  z\cdot \frac{-(2-b)  (1-e^{-q} ) e^{-q}  q'} {  ( 1-a  -(1-3a)   e^{-2q} ) }  = (1 - e^{-q} ) f(q) q',
\end{align*}
where we used the definition of $f(q)$ found in equation \eqref{f-def} and the formula for $z$ found in equation \eqref {z-eq}. 
Defining 
$$
F_2 (q) \ \dot = \ \int _0^t (1 - e^{-q}) f(q) dq ,
$$
we find 
$$
w^2 = w_0^2 + 2 F_2 (q),
$$
and similarly to $h$, $w$ remains smooth and bounded for all $0\le t \le T^c$. 
\end{proof} 

Since $p_1 = \frac12(w-h)$ and $p_2 = \frac12(h+w)$, the above proposition shows that as long as $ q \ge 0$, $p_1, p_2 <\infty$. We will next  choose the initial values for $p_1(0)$ and $p_2(0)$ to be consistent with the above proposition and which will necessarily lead to a collision time $T^c<\infty$.  
    We break our analysis into cases described by the parameters $a$ and $b$. 
    For each case, we will consider the function $p(t)=p_2^2(t)-p_1^2(t)$. From (\ref{eq6}), 
\begin{align}\label{eq110} 
        q'(t) &=p(t)\big((1-e^{-2|q(t)})+a(3e^{-2|q(t)|}-1)\big)
        \\
    \label{eq10}
      p'(t)&=2(b-2)p_1p_2(p_1^2+p_2^2)e^{-|q(t)|}+4(b-2)p_1^2p_2^2 e^{-2|q(t)|}.
    \end{align}
    We will show that $T^c$ exists  and find an upper bound by showing $q'(t)$ is bounded by a negative number so long as $q(t)$ remains non-negative. More precisely, we will prove the following lemma.
    \begin{lemma} For all $b \in \mathbb {R} $ and $a\neq 0$, there exists an initial multipeakon profile such that for some $\epsilon>0$,  $\frac{dq}{dt} < -\epsilon  <0$ for $t\in [0, T^c)$ (and hence $T^c<\infty$) or there exists a time $T^p$ such that at least one of $p_1(T^p)  =0$ or $p_2(T^p) = 0$. 
\end{lemma} 
 \begin{proof} If $\min\{T^p, T^c\} = T^p$ and $T^p<\infty$, we are done. Therefore, we will assume $p_1(t)$ and $p_2(t)$ do not equal zero. By continuity, whatever the sign of their initial data is, we may assume the solutions take as well. 
  We prove the theorem in four cases, based upon the values of $a$ and $b$, omitting the trivial case when $b=2$.  
    
  \noindent{\bf Case 1: $ a>0, b>2$.}
  We take the initial data 
  $$
   p_1(0)=\alpha + \delta,\ q_1(0)=0, \ p_2(0)= -\alpha,\ q_2(0)=\mu,
  $$
  shown in Figure \ref{fig:case1}
  and we shall prove that either a collision time exists, or one of the peakons ceases to exist.  In Figure \ref{fig:case1sln} we show the time evolution of $p_1(t)$ and $p_2(t)$ in dotted and dashed respectively, and in solid we show the time evolution of $q(t)$.  
  
By the choice of our initial data: 
$$p(0)=p_2^2(0)-p_1^2(0)= -(2\alpha \delta + \delta^2)<0, \text{ and } q(0)=q_2(0)-q_1(0)= \mu >0. $$ 
\begin{figure} [h!]
   \begin{minipage}{0.5\textwidth}
     \centering
\includegraphics[width=1\textwidth]{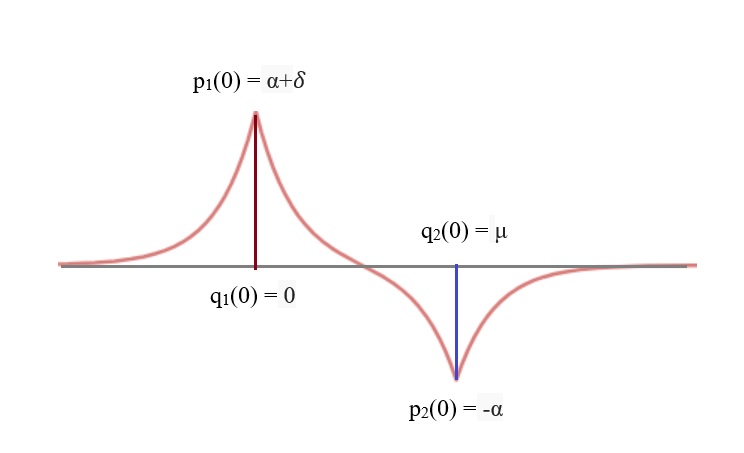}  
\caption{Case 1: Peakon-antipeakon initial data.}
\label{fig:case1} 
   \end{minipage}\hfill
   \begin{minipage}{0.5\textwidth}
     \centering
\includegraphics[width=1\textwidth]{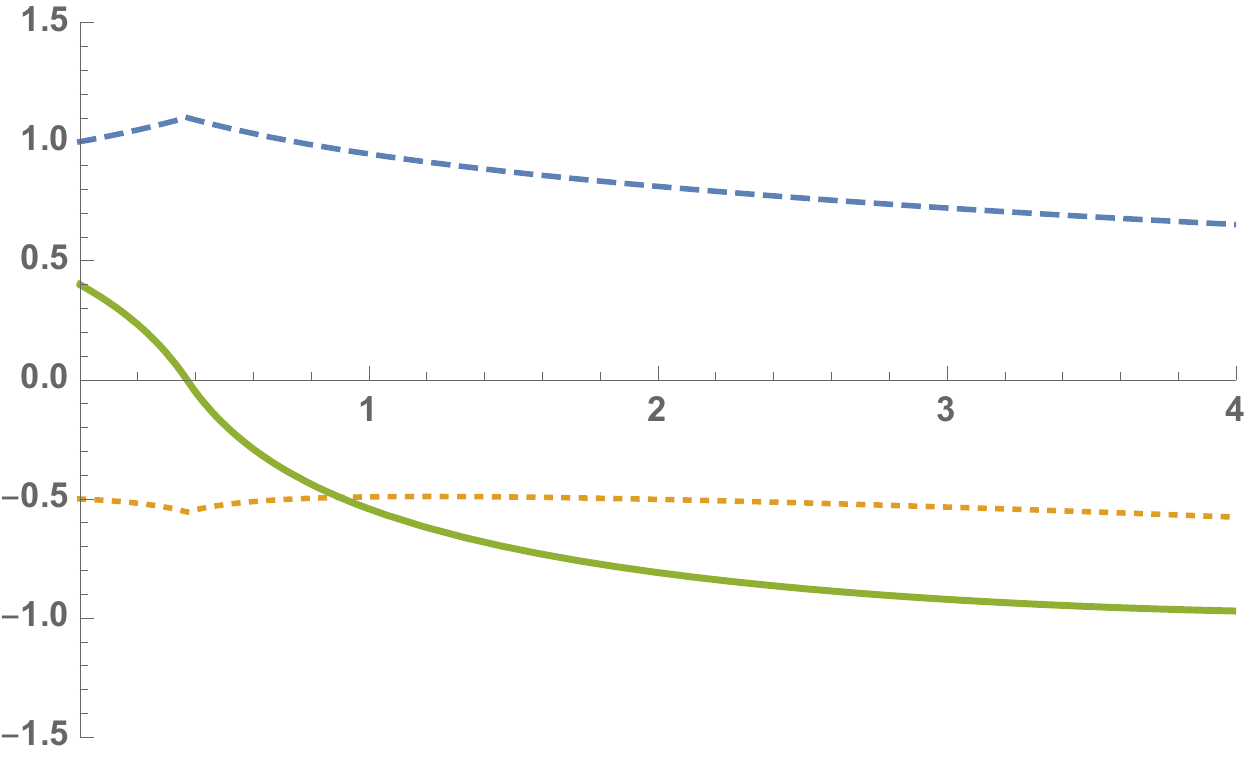}
\caption{The heights of the two peaks are dotted and dashed; $q$ is the solid line.}
\label{fig:case1sln} 
   \end{minipage}
 
\end{figure}
Since $p_2(0)=-\alpha <0$ and $p_1(0)= \alpha + \delta >0, \ p_1(0)p_2(0) <0$ and by continuity $p_1(t)p_2(t) <0$.

We will now show that for all $0\le t \le T^c$,  
$$\frac{dp}{dt}=2(b-2)p_1p_2\bigg((p_1^2+p_2^2)e^{-|q|}+2p_1p_2 e^{-2|q|}\bigg)<0.$$
 Indeed, the following calculation shows that  $\bigg((p_1^2+p_2^2)e^{-|q|}+2p_1p_2 e^{-2|q|}\bigg)>0.$
Use the fact that  $e^{-2|x|} < e^{-|x|}$    
to compute 
$$
     \bigg((p_1^2+p_2^2)e^{-|q|}+2p_1p_2 e^{-2|q|}\bigg) 
     \geq \bigg((p_1^2+p_2^2)e^{-2|q|}+2p_1p_2 e^{-2|q|}\bigg)
    =(p_1+p_2)^2e^{-2|q|}\geq 0. 
    $$
Therefore, we may now use $p(t) < p(0)<0$.  Substituting this into  equation (\ref{eq110}), we have 
$$\frac{dq}{dt} \leq  p(0)\big((1-e^{-2|q(t)| })+a(3e^{-2|q(t)|}-1)\big) .$$
The right hand side is negative, since initially $0< a \le L_a(\mu)  = ca \le 2a$, and as $q  $ decreases, 
$$
   ca \le(1-e^{-2|q(t)| })+a(3e^{-2|q(t)|}-1) \le 2a. 
$$
Therefore we compute $\epsilon$:
$$\frac{dq}{dt} \leq a p(0) = -a(2\alpha \delta + \delta^2) =-\epsilon <0 $$  Hence, either $p_1(t) = 0$, $p_2(t) = 0$, or $q(t) = 0$ in finite time.


  \noindent{\bf Case 2: $a>0, b<2 $.}
 In contrast to the peakon-antipeakon initial profile of Case 1, we take the two peakon initial profile:
  $$
   p_1(0)=\alpha +\delta,\ q_1(0)=0, \ p_2(0)= \alpha ,\ q_2(0)=\mu,
  $$
  shown in Figure \ref{fig:case2}
  and we shall prove that either a collision time exists, or one of the peakons ceases to exist. In contrast to the previous case, numerical experiments show that there are many times in which $q(t) = 0$; $T^c$ is the first time this occurs. Moreover, the numerical solution shows that the two peakons do not decouple, rather, they repeatedly leapfrog each other. 
\begin{figure} [h!]
   \begin{minipage}{0.5\textwidth}
     \centering
\includegraphics[width=1\textwidth]{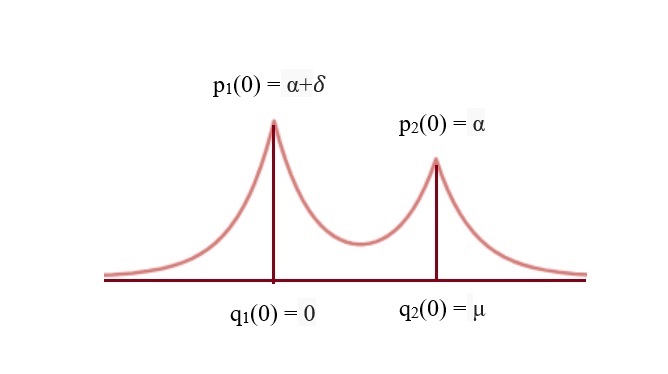}  
\caption{Case 2: 2-Peakon initial data.}
\label{fig:case2} 
   \end{minipage}\hfill
   \begin{minipage}{0.5\textwidth}
     \centering
\includegraphics[width=1\textwidth]{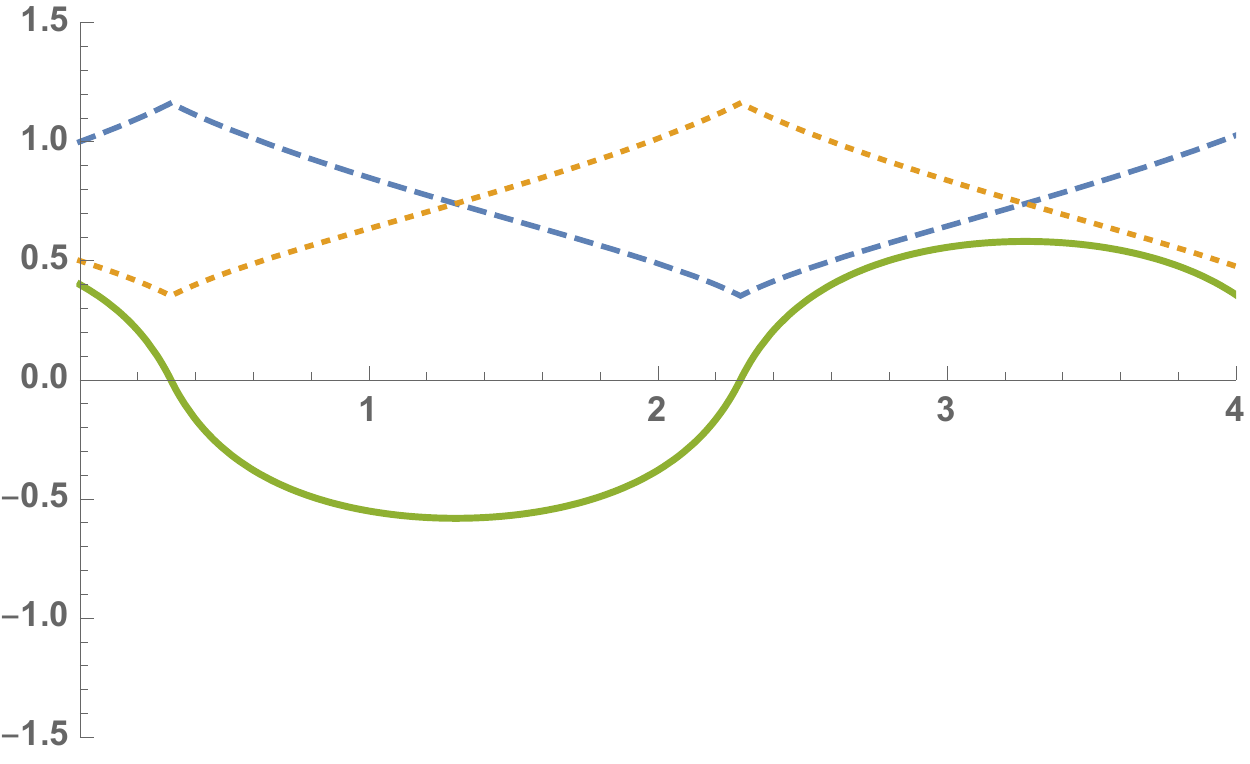}
\caption{The heights of the two peaks are dotted and dashed; $q$ is the solid line.}
\label{fig:case2sln} 
   \end{minipage}
 
\end{figure}
By the choice of our initial data 
$$p(0)=p_2^2(0)-p_1^2(0)= -(2\alpha \delta + \delta^2)<0, \ q(0)=q_2(0)-q_1(0)= \mu >0. $$  
Since $ p_1(0)p_2(0) >0$, by continuity $p_1(t)p_2(t) >0$ for $t\in [0, T^c)$. We have 
$$\frac{dp}{dt}=2(b-2)p_1p_2\bigg((p_1^2+p_2^2)e^{-|q|}+2p_1p_2 e^{-2|q|}\bigg)<0.$$
  Thus,   $p(t)<p(0) <0 \ \forall t\in [0, T^c).$ Now from the equation (\ref{eq110}),
$$\frac{dq(t)}{dt} \leq  p(0)\big((1-e^{-2|q(t)})+a(3e^{-2|q(t)|}-1)\big) \ \text{for} \ t\in [0, T^{c}],$$
and, similarly to the first case
$$\frac{dq(t)}{dt} \leq a p(0) = -a(2\alpha \delta + \delta^2)  =-\epsilon  <0.
 $$ 
 Thus, we have shown that there exists an initial profile such that $p_1(t) = 0$, $p_2(t) = 0$, or $q(t) = 0$ in finite time.


  \noindent{\bf Case 3: $a<0, b>2$.}
  Similar to Case 2, we take the two peakon initial profile:
  $$
   p_1(0)=\alpha ,\ q_1(0)=0, \ p_2(0)= \alpha +\delta,\ q_2(0)=\mu,
  $$
  shown in Figure \ref{fig:case3}. As shown in Figure \ref{fig:case3sln} the solution of this 2-peakon initial data behaves similarly to Case 2. 
\begin{figure} [h!]
   \begin{minipage}{0.5\textwidth}
     \centering
\includegraphics[width=1\textwidth]{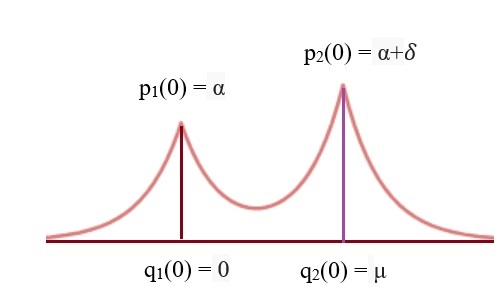}  
\caption{Case 3: 2-Peakon initial data.}
\label{fig:case3} 
   \end{minipage}\hfill
   \begin{minipage}{0.5\textwidth}
     \centering
\includegraphics[width=1\textwidth]{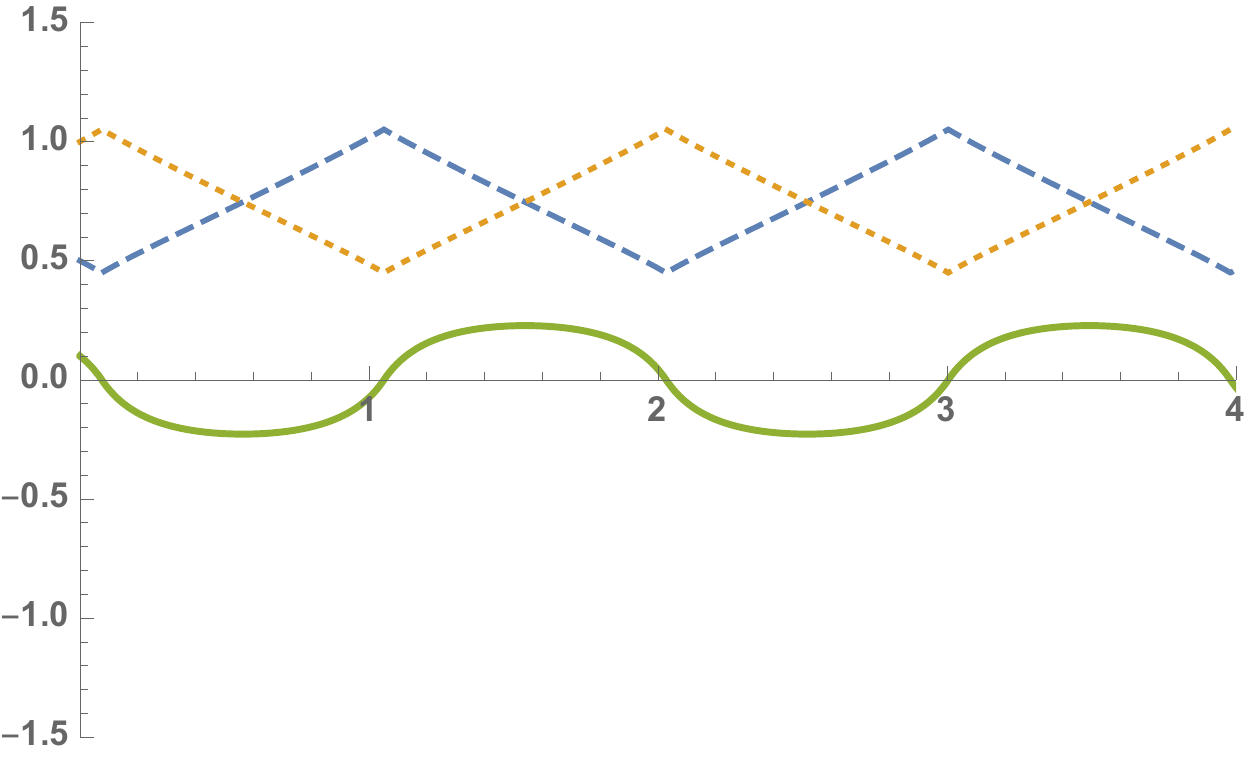}
\caption{The heights of the two peaks are dotted and dashed; $q$ is the solid line.}
\label{fig:case3sln} 
   \end{minipage}
 
\end{figure}

By the choice of our initial data 
$$p(0)=p_2^2(0)-p_1^2(0)= (2\alpha \delta + \delta^2)>0, \ q(0)=q_2(0)-q_1(0)= \mu >0. $$
By continuity $p_1(t)p_2(t) >0$  and we will assume this hold for $t \in [0, T^c)$, thus
$$\frac{dp}{dt}=2(b-2)p_1p_2\bigg((p_1^2+p_2^2)e^{-|q|}+2p_1p_2 e^{-2|q|}\bigg)>0.$$
and therefore,   $p(t)>p(0)>0 $.

Recalling the choice of $\mu$, we have $ L_a(\mu) =ca<0$ and $\Lim{q\to \ 0}L_a(q)=2a <0.$ 
Hence, 
\begin{equation}\label{l-est}
2a <L_a(q(t)) < ca <0, t\in [0, T^c) .
\end{equation}
Therefore,
using equation (\ref{eq110}) we have
$$\frac{dq(t)}{dt}=  p(t)L_a(q(t) )  \leq  p(0)\cdot ca = -\epsilon<0 .
$$
 Again, this shows that there exists an initial profile such that $p_1(t) = 0$, $p_2(t) = 0$, or $q(t) = 0$ in finite time.  


  \noindent{\bf Case 4: $a<0, b<2$.}
  Similar to Case 1, we take a  peakon-antipeakon initial profile:
  $$
   p_1(0)=-\alpha ,\ q_1(0)=0, \ p_2(0)= \alpha +\delta,\ q_2(0)=\mu,
  $$
  and as before, we will assume $T^c \le T^p$. As shown in Figure \ref{fig:case4}, this profile is not the same as in Case 1; though the behavior of the solution resembles closely Case 1 as shown in Figure \ref{fig:case4sln}.
\begin{figure} [h!]
   \begin{minipage}{0.5\textwidth}
     \centering
\includegraphics[width=1\textwidth]{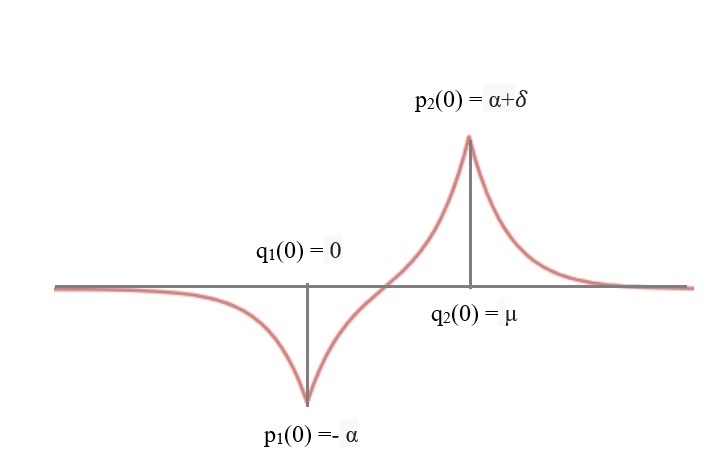}  
\caption{Case 4: Antipeakon-peakon initial data.}
\label{fig:case4} 
   \end{minipage}\hfill
   \begin{minipage}{0.5\textwidth}
     \centering
\includegraphics[width=1\textwidth]{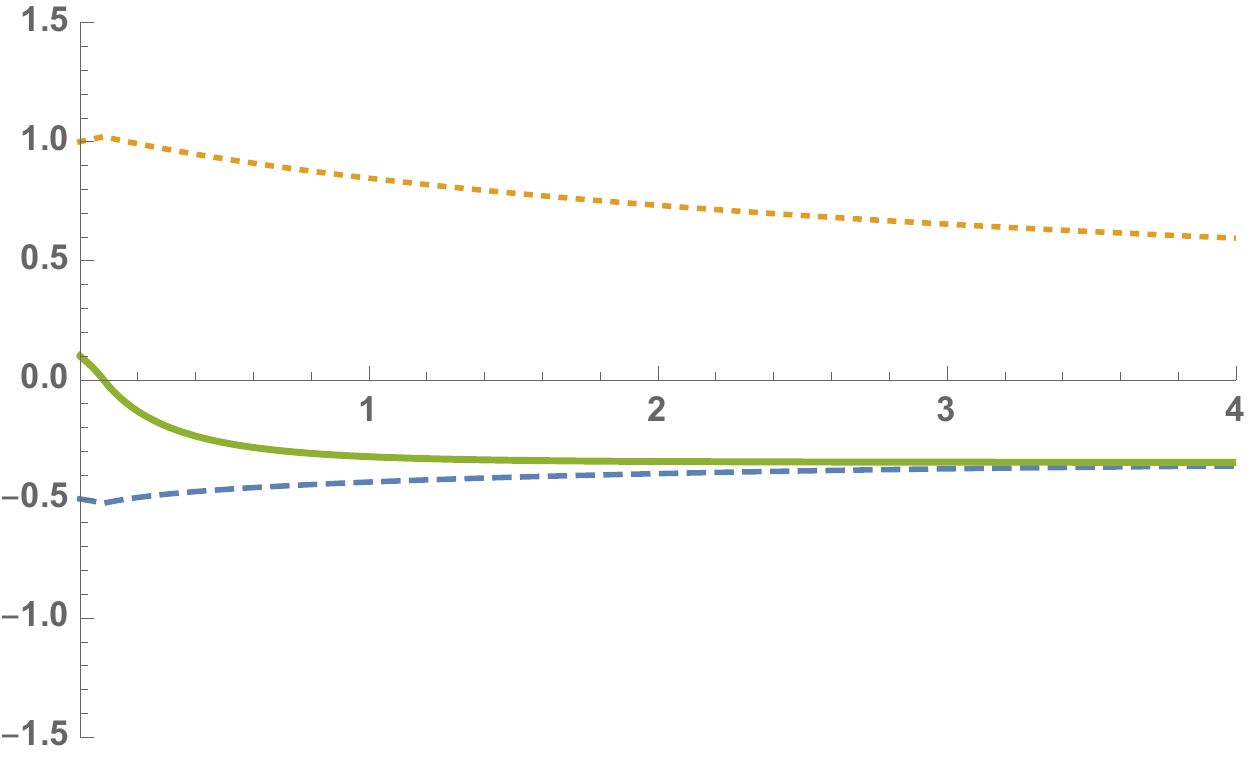}
\caption{The heights of the two peaks are dotted and dashed; $q$ is the solid line.}
\label{fig:case4sln} 
   \end{minipage}
 
\end{figure}

  By the choice of our initial data $$p(0)=p_2^2(0)-p_1^2(0)= (2\alpha \delta + \delta^2)>0, \text{ and }q(0)=q_1(0)-q_2(0)= \mu >0. $$ 
 Since $ p_1(0)p_2(0) <0$, by continuity $p_1(t)p_2(t) <0$, and therefore an argument similar to the argument presented in Case 1 shows 
 $$\frac{dp}{dt}=2(b-2)p_1p_2\bigg((p_1^2+p_2^2)e^{-|q|}+2p_1p_2 e^{-2|q|}\bigg)>0,$$
thus $p(t)>p(0)>0 \ \forall \ t\in [0, T^c).$ 
Therefore, using equation \eqref{eq110} and the estimate in inequality \eqref{l-est} we have again
$$\frac{dq(t)}{dt}= p(t)L_a(q(t)) \leq  a c p(0)   = -\epsilon<0 .$$ 
This completes the fourth case, and we have shown that for every choice of $a$ and $b$, there is an initial profile which leads to  $\min\{T^p, T^c\} < \infty$. 
\end{proof}

\subsection{Proof of Theorem \ref{main-thm}}
We are now prepared to show nonuniqueness in $H^s$ for the ab-family of equations stated in Theorem \ref{main-thm}. 
To complete the proof, we shall show that at time, $T = \min\{T^c, T^p\}$, the solution to the Cauchy problem  \eqref{eq1}-\eqref{eq1-data} with the initial profiles (depending upon $a$ and $b$) given in the previous subsection, is either a single peakon or the zero solution.

We define the collision function: 
\begin{equation}\label{eq18}
    C(x)=p^*e^{-|x-q^* |}.
\end{equation}
where if $T = T^c$, $q ^* = \Lim{t\to \ T^-} q_1(t) $ is the location of the collision and $p^{*}= \Lim{t\to \ T^-}(p_1(t)+p_2(t))$ is the magnitude of the collision. 
If $ T = T^p$, then  label $i$ and $j$ ($1$ and $2$) such that $\Lim{t\to \ T^-} p_i(t) = 0$, and $ \Lim{t\to \ T^-}p_j (t) \neq0$. We define $q ^* = \Lim{t\to \ T^-} q_j(t) $. If both $p_1$ and $p_2$ converge to zero at time $T$, since $C(x) = 0$, the choice in $q^*$ is irrelavent. 
We will now show that the solution converges to the collision function. 
\begin{lemma}
The $H^{s}$  limit of $u$, as $t$ approaches $T$ from below is C:
$$\lim_{t\to \ T^{-}}||u(t)-C||_{H^{s}} =0.$$ 
\end{lemma}
\begin{proof}
 We take the Fourier transform of $u$, and we have
$$ \widehat{u}(\xi,t)= \frac{2p_1e^{-i\xi q_1}}{1+\xi^2}+\frac{2p_2e^{-i\xi q_2}}{1+\xi^2}.
$$
Similarly, we can find the Fourier transform of $C$ as
$$ \widehat{C}(\xi)= \frac{2p^{*}e^{-i\xi q^*}}{1+\xi^2}.
$$
Calculating the $H^{s}$ norm of $u(t)-C$ gives us
$$\lim_{t\to \ T^{-}}||u(t)-C||_{H^{s}} ^2= 4\lim_{t\to \ T^{-}}\int_{\R} ({1+\xi^2})^{s-2} |p_1e^{-i\xi q_1}+p_2e^{-i\xi q_2}-p^{*}e^{-i\xi q^*}|^2d\xi.$$
We can bound the quantity inside the absolute value by $(|p_1|+|p_2|+|p^*|) \leq M < \infty.$ Let $v(\xi)=(1+\xi^2)^{s-2}\cdot M^2$ then $v$ dominates our original integral and $v$ is itself integrable when $s<3/2.$ Therefore, we may apply the Dominated Convergence Theorem and bring the limit inside the integral. 
 \begin{equation} 
    \lim_{t\to \ T^{-}}||u(t)-C||_{H^{s}} ^2=
 4\int_{\R} ({1+\xi^2})^{s-2} |p_1(T)e^{-i\xi q_1(T)}+p_2(T)e^{-i\xi q_2(T)}-p^*e^{-i\xi q^*}|^2d\xi.
    \end{equation}   
 By definition of $p^*$ and $q^*$,  the term inside the integral is zero. 
     \end{proof}

\bibliographystyle{plain}
\bibliography{ab-family-illposedness_submitted.bib}

\end{document}